\documentclass[reqno,12pt]{amsart}

\usepackage{amsmath, amsthm, amsfonts, amssymb}
\input{mathrsfs.sty}

\usepackage{amsmath}
\usepackage{amsfonts}
\usepackage{amssymb}
\usepackage{eufrak}
\usepackage{amscd}
\usepackage{amsthm}
\usepackage{amstext}
\usepackage[all]{xy}

\newcommand{\e}{\varepsilon}

   \theoremstyle{plain}%default
   \newtheorem{thm}{Theorem}%[section]
   
   \newtheorem{lem}[thm]{Lemma}
   \newtheorem{cor}[thm]{Corollary}
   \theoremstyle{definition}

   \newtheorem{example}{Example}
   \theoremstyle{remark}

\author{V. Manuilov}

\date{}

\address{Moscow Center for Fundamental and Applied Mathematics, Moscow State University,
Leninskie Gory 1, Moscow, 
119991, Russia}

\email{manuilov@mech.math.msu.su}

\thanks{The author acknowledges support by the RNF grant 23-21-00068.}

\title{Mapping graph homology to $K$-theory of Roe algebras}

\begin{document}

\begin{abstract}
Given a graph $\Gamma$, one may conside the set $X$ of its vertices as a metric space by assuming that all edges have length one. We consider two versions of homology theory of $\Gamma$ and their $K$-theory counterparts --- the $K$-theory of the (uniform) Roe algebra of the metric space $X$ of vertices of $\Gamma$. We construct here a natural map from homology of $\Gamma$ to the $K$-theory of the Roe algebra of $X$, and its uniform version. We show that, when $\Gamma$ is the Cayley graph of $\mathbb Z$, the constructed maps are isomorphisms.

\end{abstract}

\maketitle

\section{Introduction}

Given a graph $\Gamma$, one may conside the set $X$ of its vertices as a metric space by assuming that all edges have length one. There are various homology theories associated to $\Gamma$, in particular, the uniformly finite homology constructed from chains with uniformly bounded coefficients, and the Borel--Moore homology constructed from unbounded chains. The $K$-theory counterpart is the $K$-theory of the Roe algebra of the metric space $X$ of vertices of $\Gamma$, also with two versions: the uniform Roe algebra $C^*_u(X)$, and the Roe algebra $C^*(X)$. We construct here a natural map from homology of $\Gamma$ to the $K$-theory of the Roe algebra of $X$ (and its uniform version), which can be considered as a naive low-dimensional version of the coarse assembly map \cite{Yu,B-E}. In dimension 0 this construction is trivial, and in dimension 1 it is based on possibility to decompose 1-cycles into infinite paths. We show that when $\Gamma$ is the Cayley graph of $\mathbb Z$ the constructed maps are isomorphisms, although for general graphs these maps may be neither injective nor surjective.

\section{Definitions and notation}

Let $X$ be a uniformly discrete countable metric space, that is, $X$ is a countable set equipped with a metric $d$ such that $\inf_{x\neq y\in X}d(x,y)>0$. We also assume that $X$ is proper, which means that any ball in $X$ contains a finite number of points. We write $l^2(X)$ (resp., $C_c(X)$) for the Hilbert space of square summable complex-valued functions on $X$ (resp., for the algebra of complex-valued functions with finite support on $X$).

For a Hilbert space $H$ we write $\mathbb B(H)$ (resp., $\mathbb K(H)$) for the algebra of all bounded (resp., all compact) operators on $H$. A function $f\in C_c(X)$ defines an operator $M_f\in\mathbb B(l^2(X)\otimes H)$ by $M_f(\varphi\otimes\xi)=f\varphi\otimes\xi$, $\varphi\in l^2(X)$, $\xi\in H$.

Recall the definition of the Roe algebra of $X$ \cite{Roe}.

An operator $T\in\mathbb B(l^2(X)\otimes H)$ is {\it locally compact} if the operators $TM_f$ and $M_fT$ are compact for any $f\in C_c(X)$. It has {\it finite propagation} if there exists some $R>0$ such that $M_fTM_g=0$ whenever the distance between the supports of $f,g\in C_c(X)$ is greater than $R$. If $\dim H=\infty$ then the norm closure of the $*$-algebra of locally compact, finite propagation operators on $l^2(X)\otimes H$ is called the {\it Roe algebra} of $X$ and is denoted by $C^*(X)$. If $\dim H=1$ then local compacness is redundant, and the norm closure of the $*$-algebra of finite propagation operators on $l^2(X)$ is called the {\it uniform Roe algebra} of $X$ and is denoted by $C^*_u(X)$.  
The former is better related to elliptic operators and index theory, while the latter is more tractable (e.g. is unital). Embedding $\mathbb C\subset H$ for an infinitedimensional Hilbert space $H$ induces an embedding $i:C^*_u(X)\to C^*(X)$.

We fix the orthonormal basis on $l^2(X)$ consisting of delta functions $\delta_x$, $x\in X$ defined by $\delta_x(y)=\left\lbrace\begin{array}{cl}1,&\mbox{if\ }y=x;\\0,&\mbox{if\ }y\neq x.\end{array}\right.$ We write $H_x$ for $\delta_x\otimes H\subset l^2(X)\otimes H$ and $P_{H_x}$ for the projection onto $H_x$. Given an operator $T\in\mathbb B(l^2(X)\otimes H)$, we write $T_{xy}$ for $P_{H_x}T|_{H_y}:H_y\to H_x$. This allows to view $T$ as a matrix $(T_{xy})_{x,y\in X}$ with operator-valued entries.

Let $\Gamma$ be a graph with the set of vertices $V(\Gamma)=X$ and the set of edges $E(\Gamma)$. We shall use also the notation $V(\Gamma)=\Gamma^0$ and $E(\Gamma)=\Gamma^1$. We consider $\Gamma$ as a one-dimensional cellular space. In order to introduce a metric on $\Gamma$ we assume that each edge is isometric to $[0,1]$ with the standard metric (in particular, it has length one). Then the distance between two points is defined as the infimum of lengths of all paths that connect these points. This metric can be restricted to the set $X$ of vertices, makking the latter a metric space too. The metric space $\Gamma$ (or $X$) is proper iff $\Gamma$ is locally compact iff each vertex has a finite number of adjacent edges.

There is a standard way back from a discrete metric space to graphs: let $X$ be the set of vertices of a graph. Given $\alpha>0$, we connect two points $x,y\in X$ by an edge if $d(x,y)\leq\alpha$. This gives a family of graphs $\Gamma_\alpha$, $\alpha\in(0,\infty)$. 

We shall assume that the graphs in the paper are oriented, i.e. each edge $e\in E(\Gamma)$ goes from its source vertex $s(e)$ to its target vertex $t(e)$. This orientation can be chosen arbitrarily, our results do not depend on it.

Recall that the group $C_k^{BM}(\Gamma)$ of $k$-dimensional Borel--Moore chains is the abelian group of all formal sums $\sum_{x\in\Gamma^k}\lambda_x\cdot x$, $k=0,1$, $\lambda_x\in\mathbb Z$. The standard differential $d:C_1^{BM}(\Gamma)\to C_0^{BM}(\Gamma)$ is defined by $d(e)=t(e)-s(e)$. The 0-th and the 1-st Borel--Moore homology of $\Gamma$ are given by
$$
H_0^{BM}(\Gamma)=C_0^{BM}(\Gamma)/d C_1^{BM}(\Gamma),
$$
$$
H_1^{BM}(\Gamma)=\{\gamma\in C_1^{BM}(\Gamma):d(\gamma)=0\}.
$$
Details on Borel--Moore homology can be found in \cite{B-M,Spanier} and in \cite{Ranicki}, Appendix A. 

Along with Borel--Moore homology, one can use the uniformly finite homology \cite{Novak-Yu}. A chain $\sum_{x\in\Gamma^k}\lambda_x\cdot x\in C_k^{BM}(\Gamma)$, $k=0,1$, is uniformly finite if there exists $K>0$ such that $|\lambda_x|<K$ for any $x\in\Gamma^k$. Let $C_k^{uf}(\Gamma)\subset C_k^{BM}(\Gamma)$ be the set of uniformly finite chains. Recall that a discrete metric space has bounded geometry if, for every $r>0$, there exists $K_r>0$ such that any ball of radius $r$ has less than $K_r$ points. If $X$ has bounded geometry then for each $r>0$ there exists $K_r$ such that the valence of each $x\in X$ (i.e. the number of edges adjacent to $x$) is less than $K_r$. In this case $d$ maps $C_1^{uf}(\Gamma)$ to $C_1^{BM}(\Gamma)$, hence one can define the uniformly finite homology $H_k^{uf}(\Gamma)$ together with the canonical map $j:H_k^{uf}(\Gamma)\to H_k^{BM}(\Gamma)$, $k=0,1$. 

Relation between the two homology groups can be illustrated by the following example. Let $c=\sum_{x\in X}x\in C_0^{uf}(\Gamma)$. It is shown in Proposition 7.3.4 of \cite{Novak-Yu} that if $\Gamma$ is the Cayley graph of a finitely generated group then $[c]\neq 0$ in $H_0^{uf}(\Gamma)$ iff the group is amenable. On the other hand, it was shown in Proposition 11.1.3 of \cite{Geoghegan} (cf. also \cite{K-M}) that if $\Gamma$ is path connected then $[c]=0$ in $H_0^{BM}(\Gamma)$. 

\section{map from $H_0$ to $K_0$}

Let $c=\sum_{x\in\Gamma}c_x\cdot x\in C_0^{BM}(\Gamma)$, $c_x\in\mathbb Z$. Set $(f_c(x),g_c(x))=\left\lbrace\begin{array}{cl}(p_{c_x},0),&\mbox{if\ }c_x\geq 0;\\(0,p_{-c_x}),&\mbox{if\ }c_x\leq 0,\end{array}\right.$
where $p_k=\sum_{i=1}^ke_i\in\mathbb K(H)$ is the projection onto the first $k$ vectors of the basis $\{e_i\}_{i\in\mathbb N}$ of $H$. Then $f_c,g_c$ are projection-valued maps from $\Gamma$ to $\mathbb K(H)$, hence diagonal elements of $C^*(\Gamma)$. Then $[f_c]-[g_c]\in K_0(C^*(\Gamma))$. 

\begin{lem}
If $c$ is uniformly finite then $f_c,g_c\in M_n(C^*_u(X))$ for some $n$.

\end{lem}
\begin{proof}
If $\max_{x\in X}|c_x|\leq K$ then $f_c(x),g_c(x)\in M_K\subset\mathbb K(H)$, hence $f_c,h_c\in M_K(C^*_u(X))$.
\end{proof}

\begin{lem}
If $c=d(\gamma)$, $\gamma\in C_1^{BM}(\Gamma)$, then $[f_c]=[g_c]$.

\end{lem}
\begin{proof}
Let $\gamma=\sum_{e\in\Gamma^1}\gamma_e e$, $\gamma_e\in\mathbb Z$. By changing orientation we may assume that each $\gamma_e$ is non-negative. Also for simplicity consider a new oriented graph $G$, which has the same vertices as $\Gamma$, but each edge $e$ is replaced by $\gamma_e$ edges going from $s(e)$ to $t(e)$. For $x\in G^0=\Gamma^0$, let $i(x)$ and $o(x)$ denote the number of ingoing and outgoing edges in $G$. Then $i(x)=\sum_{e\in\Gamma^1,t(e)=x}\gamma_e$, $o(x)=\sum_{e\in\Gamma^1,s(e)=x}\gamma_e$.

We have $c_x=\sum_{e\in\Gamma^1,t(e)=x}\gamma_e-\sum_{e\in\Gamma^1,s(e)=x}\gamma_e=i(x)-o(x)$. Let $f'(x)=p_{i(x)}$, $g'(x)=p_{o(x)}$, then $[f]-[g]=[f']-[g']$, and it remains to show that the projections $f'$ and $g'$ are equivalent in $C^*(X)$.

Write $H_x=A_x\oplus B_x\oplus(A_x\oplus B_x)^\perp$, where $A_x$ and $B_x$ are subspaces with orthogonal bases $\{a_{x,e}:e\in G^1,s(e)=x\}$ and $\{b_{x,e}:e\in G^1,t(e)=x\}$ respectively. Then $\dim A_x=i(x)$, $\dim B_x=o(x)$. Define $V:\oplus_{x\in X}H_x\to\oplus_{x\in X}H_x$ by $Va_{x,e}=b_{t(e),e}$ and $V\xi=0$ for any $\xi\in H_x\ominus A_x$. Then $V$ is a partial isometry, and $V^*V$, $VV^*$ are projections onto $\oplus_{x\in X}A_x$ and $\oplus_{x\in X}B_x$ respectively. Clearly, $f'$ is equivalent to $V^*V$, and $g'$ is equivalent to $VV^*$. 
\end{proof}

\begin{cor}
The map $\varphi_0:H_0^{BM}(\Gamma)\to K_0(C^*(X))$ given by $\varphi_0(c)=[(f_c,g_c)]$ is well defined. Its restriction to uniformly finite chains induces a map $\varphi'_0:H_1^{uf}(\Gamma)\to K_0(C^*_u(X))$.

\end{cor}

\section{Map from $H_1$ to $K_1$}

Now let us construct a map $\varphi_1:H^{BM}_1(\Gamma)\to K_1(C^*(X))$. Let $\gamma=\sum_{e\in\Gamma^1}\gamma_e e\in C_0^{BM}(\Gamma)$ satisfy $d\gamma=0$. The latter means that 
\begin{equation}\label{i=o}
i(x)=\sum_{e\in\Gamma^1,t(e)=x}\gamma_e=\sum_{e\in\Gamma^1,s(e)=x}\gamma_e=o(x). 
\end{equation}
Denote the number in (\ref{i=o}) by $n_x$.

As before, we assume that all $\gamma_e$ are non-negative and pass from $\Gamma$ to $G$. For $x\in\Gamma^0$, the sets $I(x)=\{e\in G^1:t(e)=x\}$ and $O(x)=\{e\in G^1:s(e)=x\}$ have the same cardinality $n_x$, hence we may fix a bijection 
\begin{equation}\label{alpha}
\alpha_x:I(x)\to O(x). 
\end{equation}

Note that as $X$ is proper, the set $G^1$ is countable. Then we may assume that there is an orthonormal basis $\{\e_e\}_{e\in G^1}$ of $H$ numbered by the edges of $G$.

Define a map $f=f_\alpha:X\times\ G^1\to X\times G^1$ by 
$$
F(x,e)=(y(x,e),f(x,e))=\left\lbrace\begin{array}{cl}(t(\alpha_x(e)),\alpha_x(e)),&\mbox{if\ }t(e)=x;\\
(x,e),&\mbox{otherwise,}\end{array}\right.
$$ 
and an operator $U_\gamma=U_\gamma^{(\alpha)}$ on $l^2(X)\otimes H$ by 
$$
U(\delta_x\otimes\e_e)=\delta_{y(x,e)}\otimes \e_{f(x,e)}.
$$ 

\begin{lem}
$U_\gamma$ is unitary, and $U_\gamma-1\in C^*(X)$.

\end{lem}
\begin{proof}
The map $F$ is a bijection, hence $U_\gamma$ is unitary. If $x,y\in X$ are not adjacent in $G$ then $(U_\gamma)_{xy}=0$.  For any $x\in X$ there are only finitely many edges with $s(e)=x$, hence $y(x,e)\neq x$ for finitely many edges, i.e. $(U_\gamma-1)_{xy}$ is of finite rank.
\end{proof}

The above construction depends on choice of the bijections $\{\alpha_x\}_{x\in X}$ (\ref{alpha}). Let $\{\alpha_x\}_{x\in X}$, $\{\beta_x\}_{x\in X}$ be two sets of bijections between $I(x)$ and $O(x)$, and let $U_\gamma^{(\alpha)}$ and $U_\gamma^{(\beta)}$ be the corresponding unitaries.

\begin{lem}
The unitaries $U_\gamma^{(\alpha)}$ and $U_\gamma^{(\beta)}$ lie in the same connected component of the unitary group of the unitalization of $C^*(X)$.

\end{lem}
\begin{proof}
Set $U'(\delta_x\otimes\e_e)=\left\lbrace\begin{array}{cl}\delta_{t(\alpha_x(e))}\otimes\e_{\beta_x(e)},&\mbox{if\ }t(e)=x;\\\delta_x\otimes\e_e,&\mbox{otherwise.}\end{array}\right.$

For each $x\in X$, Let $v_x$ be a unitary operator on $\delta_x\otimes H$ such that it interchanges $\alpha_x(e)$ with $\beta_x(e)$ when $t(e)=x$, and $v_x(\delta_x\otimes \e_e)=\delta_x\otimes\e_e$ when $t(e)\neq x$. Let $V$ be the diagonal operator on $l^2(X)\otimes H$ with diagonal entries equal to $v_x$, $V=\oplus_{x\in X}v_x$. Then $U'=V^*U_\gamma^{(\alpha)}V$. As $v_x-1$ has finite rank, $V-1\in C^*(X)$, and there exists a homotopy that connects $V$ to 1, which lies in the unitalization of $C^*(X)$. Thus, $[U_\gamma^{(\alpha)}]=[U']$. 

For each $e\in G^1$, there are finitely many vertices $x\in X$ such that $t(e)=x$. Let $w_e:l^2(X)\otimes\e_e\to l^2(X)\otimes H$ be an operator that intertwines $t(\alpha_x(e))$ with $t(\beta_x(e))$ when $x=t(e)$, and $w_e(x)=x$ when $x\neq t(e)$.   
Set $W=\oplus_{e\in G^1}w_e$. It is unitary, and $U_\gamma^{(\beta)}=W^*U'W$. Also $W-1\in C^*(X)$ as it is locally compact of finite propagation. As $W$ is an infinite direct sum of unitaries, each of which is unit plus a finite rank operator, there is a homotopy that connects $W$ and 1, hence $[U']=[U_\gamma^{(\beta)}]$.
\end{proof}

\begin{cor}
The map $\varphi_1:H_1^{BM}(\Gamma)\to K_1(C^*(X))$ given by $\varphi_1(\gamma)=[U_\gamma]$ is well defined.

\end{cor}

\begin{lem}
If $X$ has bounded geometry and $\gamma$ is uniformly finite then there exists $n\in\mathbb N$ and a unitary $\tilde U_\gamma\in M_n(C^*_u(X))$ such that $[\tilde U_\gamma]=[U_\gamma]$ in $K_1(C^*(X))$. 

\end{lem}
\begin{proof}
Bounded geometry of $X$ and uniform boundedness of $\gamma$ imply that there exists $n\in\mathbb N$ such that the valence of any vertex $x\in X$ (the number of edges $e\in G^1$ such that $s(e)=x$ or $t(e)=x$) does not exceed $n$. As the set $G^1$ is contaable, we may number them by integers, i.e. $G^1=\{e_i\}_{i\in\mathbb N}$. For each $x\in X$ define a unitary operator $T_x:H_x\to H_x$ by the following permutation of edges: $T_x$ interchanges the first $i(x)$ edges with the edges in $I(x)$, and keeps other edges ($\{e_1,\ldots,e_{i(x)}\}\cap G^1\setminus I(x)$) unchanged. Clearly, $T_x-1$ is of finite rank for each $x\in X$, hence $T=\oplus_{x\in X}T_x$ lies in the unitalization of $C^*(X)$. Set $\tilde U_\gamma=T^* U_\gamma T$. 
If $i>n$ then $T^*U_\gamma T(\delta_x\otimes\e_{e_i})=T^*U_\gamma(\delta_x\otimes e)=T^*(\delta_x\otimes\e_e)=\delta_x\otimes\e_{e_i}$, where $e\notin I(x)$, therefore $\tilde U_\gamma-1\in M_n(C^*_u(X))$. As $T$ lies in the component of 1, $[\tilde U_\gamma]=[U_\gamma]$.
\end{proof}

\begin{cor}
The map $\varphi_1$ restricted to uniformly finite chains induces a map $\varphi'_1:H_1^{uf}(\Gamma)\to K_1(C^*_u(X))$.
\end{cor}
\begin{proof}
Clearly, the class of $\tilde U_\gamma$ does not depend on the choice of the ordering of the sets $I(x)$.
\end{proof}

\begin{thm}
The maps $\varphi_k$, $k=0,1$, are well defined and can be organized into the commuting diagram
\begin{equation}\label{diagram}
\begin{xymatrix}{
H_k^{uf}(\Gamma)\ar[r]^-{\varphi_k}\ar[d]_-{j}&K_k(C^*_u(X))\ar[d]^-{i}\\
H_k^{BM}(\Gamma)\ar[r]^-{\varphi_k}&K_k(C^*(X)).
}\end{xymatrix}
\end{equation}

\end{thm}

\section{Examples}

\begin{example} 
Let $X=\mathbb Z\subset\mathbb R$  with the standard metric, and $\Gamma=\mathbb R$ with the edges $[n,n+1]$, $n\in\mathbb Z$. This is the Cayley graph of the group $\mathbb Z$.

Denote by $C_b(\mathbb Z,\mathbb Z)$ (resp., by $C(\mathbb Z,\mathbb Z)$) the set of bounded (resp., of all) maps from $\mathbb Z$ to $\mathbb Z$ with the pointwise abelian group operation. We write $\{n_i\}_{i\in\mathbb Z}$ for a map $i\mapsto n$ in $C(\mathbb Z,\mathbb Z)$. The chain groups $C_0^{uf}(\Gamma)$ and $C_0^{BM}(\Gamma)$ are isomorphic to $C_b(\mathbb Z,\mathbb Z)$ and to $C(\mathbb Z,\mathbb Z)$ respectively. For $\{n_i\}_{i\in\mathbb Z}\in C(\mathbb Z,\mathbb Z)$, set $S(\{n_i\})=\{n_{i+1}\}$. If $c=d\gamma$, $c=\sum_{i\in\mathbb Z}c_i i$, then, clearly, $S(\{c_i\})=\{c_i\}$. If $S(\{c_i\})=\{c_i\}$ for a chain $c$ then $c=d\gamma$ for $\gamma=\sum_{i\in\mathbb Z}c_i[i,i+1]$. Therefore, $H_0^{uf}(\Gamma)=C_b(\mathbb Z,\mathbb Z)/1-S$, $H_0^{BM}(\Gamma)=C(\mathbb Z,\mathbb Z)/1-S$, and the latter group vanishes, as any two-sided sequence can be written in the form $\{n_i\}-S(\{n_i\})$ for some unbounded sequence $\{n_i\}_{i\in\mathbb Z}$.  

If $\gamma=\sum_{i\in\mathbb Z}\gamma_i[i,i+1]\in C_1^{BM}(\Gamma)$ satisfies $d\gamma=0$ then $\gamma_i=\gamma_j$ for any $i,j\in\mathbb Z$, hence $H_1^{BM}(\Gamma)\cong\mathbb Z$. Similarly, $H_1^{uf}(\Gamma)\cong\mathbb Z$. 

Roe algebras of groups are isomorphic to cross products \cite{Roe}: $C^*_u(G)\cong l^\infty(G,\mathbb C)\rtimes G$, $C^*(G)\cong l^\infty(G,\mathbb K)\rtimes G$, where $\mathbb K$ denotes the algebra of compact operators. For $G=\mathbb Z$ one can use the Pimsner--Voiculescu exact sequence \cite{P-V} to calculate $K$-groups. As $K_0(l^\infty(G,\mathbb C))\cong C_b(\mathbb Z,\mathbb Z)$, $K_0(l^\infty(G,\mathbb K))\cong C(\mathbb Z,\mathbb Z)$, $K_1(l^\infty(G,\mathbb C))=K_0(l^\infty(G,\mathbb K))=0$, we have that $K_0(C^*_u(\mathbb Z))=C_b(\mathbb Z,\mathbb Z)/1-S$, $K_0(C^*(\mathbb Z))=C(\mathbb Z,\mathbb Z)/1-S$, $K_1(C^*_u(\mathbb Z))=K_1(C^*(\mathbb Z))\cong\mathbb Z$. 

For both $k=0$ and $k=1$ the maps $\varphi_k$ in the diagram (\ref{diagram}) are isomorphisms between $H_k^{uf}(\Gamma)$ and $K_k(C^*_u(\mathbb Z))$ and between $H_k^{BM}(\Gamma)$ and $K_k(C^*(\mathbb Z))$. 

\end{example}

\begin{example} 
Let $X=\mathbb Z$ with the standard metric, and let $\Gamma$ be the graph with the vertices in $X$ and with no edges. The $K$-theory groups are the same. As there are no edges, $H_1^{uf}(\Gamma)=H_1^{BM}(\Gamma)=H_0^{BM}(\Gamma)=0$, and $H_0^{uf}(\Gamma)\cong C_b(\mathbb Z,\mathbb Z)$. All the maps $\varphi_k$ in the diagram (\ref{diagram}) are trivial, except the map $\varphi_0:H_0^{uf}(\Gamma)\to K_0(C^*_u(\mathbb Z))$, which is the quotient map $C_b(\mathbb Z,\mathbb Z)\to C_b(\mathbb Z,\mathbb Z)/1-S$. 

\end{example}

\begin{example} 
Note that the range of the map $\varphi_0$ lies in the range of the map $l^\infty(X)\to C^*_u(X)\subset C^*(X)$, therefore it cannot contain the higher dimension classes, hence $\varphi_0$ is not surjective for e.g. the Cayley graph of the group $\mathbb Z^2$ (cf. \cite{L-W}).

\end{example}

%\medskip

%Given a metric space $X$ and the family of graphs $\Gamma_\alpha$, $\alpha\in(0,\infty)$, 
%the canonical inclusion $\Gamma_\alpha\to\Gamma_\beta$ for $\alpha<\beta$ allows to define the direct limit groups %$\lim_{\alpha}H_k^{uf}(\Gamma_\alpha)$ and $\lim_{\alpha\to\infty}H_k^{BM}(\Gamma_\alpha)$, and the corresponding direct %limit homomorphisms $\tilde\varphi_k$ to $K$-theory, $k=0,1$. 

\end{document}